\documentclass[11pt, a4paper]{article}
\usepackage{authblk}
\usepackage{amsthm}  
\usepackage{amssymb}
\usepackage{amsmath}
\usepackage{graphicx}
\usepackage{hyperref}  
\usepackage{enumerate}
\usepackage{amssymb,color}

\newtheorem{theorem}{Theorem}[section]

\numberwithin{theorem}{section}
\numberwithin{equation}{section} 
\newtheorem{definitionN}[theorem]{Definition}
\newtheorem{lemmaN}[theorem]{Lemma}
\newtheorem{exampleN}[theorem]{Example}
\newtheorem{unnamed}[theorem]{}

\newtheorem{remarkN}[theorem]{Remark}

\definecolor{MyGreen}{RGB}{29,162,55}

\begin{document}

\title{\bf 
  The Extension of the Desargues Theorem, The Converse,
  Symmetry and Enumeration
}

\author{Aiden A. Bruen
}
\affil{
  Adjunct Research Professor,\\
  School of Mathematics and Statistics,\\
  Carleton University,\\
  Ottawa, Ontario, Canada,\\
  {\tt abruen@math.carleton.ca}
}

\date{}
\maketitle




%
%
%
\section{Introduction}
The fundamental building blocks of projective geometry 
are the theorems of Pappus of Alexandria, living in the 
fourth century A.D., and the theorem of Desargues published 
in Paris by Girard Desargues a French architect, engineer 
and mathematician in 1629.

The celebrated Desargues perspective theorem in the plane, 
over any field or skew field, states that when two triangles 
are in perspective the meets of corresponding sides are collinear. 
The theorem and the convserse
can be proven by using 
coordinates or by invoking the principle of duality in projective geometry.

In Coxeter (\cite{coxeter}) the author writes: 
\begin{itemize}
\item[]
Is it possible to develop a geometry having no circles, 
no distances, no angles, no intermediacy (or ``betweenness''), 
and no parallelism?
Surprisingly, the answer is Yes: what remains is projective geometry: 
a beautiful and intricate series of propositions, simpler than 
Euclid’s but not too simple to be interesting... 
The original motivation for this kind of geometry 
came from the fine arts. It was in 1425 that the 
Italian architect Brunelleschi began to discuss the 
geometrical theory of perspective which was consolidated 
into a treatise by Alberti a few years later. 
\end{itemize}
Vanishing points in drawing become mathematical points at infinity, 
all lying on a line at infinity.  For an elaboration of this 
we refer to ``Art and Geometry'' by William Ivins, 
a former curator of the Metropolitan Museum of Art~\cite{ivins}.

Crannell and Douglas~\cite{crannell}
and Lord~\cite{lord} also provide interesting and related
background.

A brief overview of the connection of the 
Desargues theorem with axiomatics is as follows. 
We start with a point $V$ and line $l$ which may or may not 
be on $V$ in a projective plane and study a possible 
``central collineation'' $T$ that fixes $V$, all lines 
through $V$ and all points on $l$.

Let $A$ be a point unequal $V$ and not on $l$. 
Then $T(A)$ has to be a point $D$ on the line $VA$. 
Let $B$ be chosen with $B$ not on $l$ or $VA$. 
$T(B)$ is a point $E$ on $VB$.
Using the corresponding pair $\{A,D\}$ we see that $E$ 
must be the intersection of $BV$ and $DZ$ where $AB$ meets $l$ in $Z$.
Similarly, let $C$ to be a point on $VC$ where the lines 
$VA$, $VB$, $VC$ are distinct.
To find $T(C)=F$ we can use the pair $\{A,D\}$.  
Then $F$ is the point of intersection of $VC$ and $DW$,
where $AC$ meets $l$ in $W$.  We can also use the pair $\{B,E\}$.  
Then $F$ must be the point where $CV$ and $EU$ 
meet where $BC$ meets $l$ in $U$. Both constructions must give 
the same answer for $F$.
Thus, for the central collineation to exist, the Desargues theorem 
must hold for the triangles $ABC$, $DEF$, in perspective from $V$,
since the intersections of the corresponding sides must all lie 
on the line $l$.

Conversely, if for all triangles $ABC$, $DEF$ in perspective 
from $V$, the intersections of corresponding lines lie on a 
line then there exists a such a central collineation $T$,
fixing $V$, all lines on $V$ and a line $l$ pointwise, 
where $l$ contains the intersections of corresponding lines. 
This suggests an indirect, and involved, proof of the 
extended Desargues theorem in any dimension.

However, the emphasis here is on synthetic reasoning and the 
resulting configurations. From~\cite[p.\ 141]{rota}: 
``The proof of Desargues’ theorem of projective geometry 
comes as close as a proof can to the Zen ideal.  
It can be summarized in two words: `I see'.''
The topic of geometrical configurations has undergone a resurgence 
in recent years: see for example
Conway and Ryba~\cite{conway},
Luotoniemi~\cite{luotoniemi}
and
G\'evay~\cite{gevay}.

In this paper, inter alia, we show that the analogue of the 
Desargues theorem holds in all dimensions over infinite and 
finite fields of sufficiently large order. The result is that 
if two simplexes with no common points or faces are in perspective 
from a point then the intersections of corresponding $t$-spaces 
are $t-1$ spaces, all lying in a hyperplane $H$ for $t=1,2,....n-1$.
The simple, synthetic proof of this result, and the converse, 
in Section~\ref{section:DesInNDim}, valid in all dimensions, 
only assumes that pairs of corresponding edges meet in a point. 
The second proof is based on arcs as in~\cite{BruenBruenMcQuillan}. 
In the planar case, although several authors use a 5-point in 
3 dimensions the proofs can still be quite complicated. 
The elegant proof in~\cite{conway} uses a different approach -
see Section~\ref{section:fourthProof}.

Regarding previous work we mention that in 
1916~\cite[p.\ 43-44]{veblenYoungVol1} the authors offer a 
proof for $n=3$. The accompanying diagram for this 3-dimensional 
case includes 15 points and many lines. There are many interesting 
recent papers, too numerous to cite, in the general area.  
The paper by G\'evay includes several references. The 
papers by Luotenemi on models of important configurations 
such as the Desargues configuration and the double six
are informative 
and very helpful for visualization.

A potential problem with synthetic proofs in geometry is 
unanticipated coincidences of points or collinearities of lines. 
Concerning the classical Hessenberg theorem, showing that 
Pappus implies Desargues, we have the following in Pedoe - 
see Introduction (\cite{pedoe}),
``It should interest those who may be disposed to believe 
that all outstanding problems in classical projective geometry 
have been solved to note that Pickert in 
his {\em Projective Ebenen} (Berlin 1955) lists eight defective 
versions of the classical Hessenberg theorem. 
Two of these defective proofs are by Hessenberg himself''.

In~\cite{veblenYoungVol1}, for the case $n=3$, 
although they do not state it, the authors assume 
that the two simplexes do not share a face. 
If the simplexes have a face in common the above result 
is false for the case $t=n-1$.  Merely assuming perspectivity 
from a point is not enough.

In \cite[p.\ 54, problem 26]{{veblenYoungVol1}}, 
the writers enunciate the extension of Desargues theorem 
in $n$ dimensions. No proof is offered. They cite research 
by A.~Cayley 
``Sur quelque th\'eor\`eme de la g\'eom\'etrie de position' , 
Crelle’s Journal, vol.~31, (1846): Collected papers, 
Volume 1, p.~317, and also a paper by G.~Veronese, 
``Behandlung der projectivischen Verh\"altnisse der R\"aume 
von verschiedenen Dimensionen durch das Princip des 
Prjjicirens und Schneidens'', {\em Math Annalen}, 
vol.~34, 1889 together with a paper by W.~B.\ Carver~\cite{carver}.
In this paper the author states that Cayley's paper 
describes sections by the plane or 3-dimensional space of 
the complete $n$-point in higher dimensions and that the 
Veronese paper is concerned with ``the configurations 
thus obtained in $r$ dimensions''. In a footnote Carver 
refers to the Desargues theorem in the plane as an 
``incidental occurrence of these configurations''. 
On line~7 the author writes: ``both Cayley and Veronese 
state that these same configurations can also be obtained 
as projections of higher-dimensional figures''. 

It appears that no proof of the extended Desargues Theorem 
has ever been written down.
Apart from the sketch of the proof above, using 
homogeneous coordinates, we offer two further, 
different proofs of the extended theorem, valid for all $n$.

In the 1964 paper by S.~R.\ Mandan~\cite{mandan} 
entitled ``Desargues Theorem in $n$-space'' the author 
describes a result concerning two simplexes of size $n+2$ in 
$n+1$-space which, between them, span a $2n$-space or a 
$2n+1$ space.  For $n=1$ the Desargues theorem in the plane 
follows from one of the cases. However, the extended Desargues theorem 
deals with two simplexes in an $n$-space which therefore 
span just that $n$-space. 
Thus, the result in~\cite{BruenThasBlokhuis}
does not apply to the result here.

In Rota~\cite[p.\ 145]{rota} referring to the first of the 
6 volumes of H.~F.\ Baker's {\em Principles of Geometry} 
he writes: ``After an argument that runs well over 
one hundred pages, Baker shows that beneath the statement 
of Desargues’ theorem, another far more interesting 
geometric structure lies concealed. This structure is 
nowadays called the Desargues configuration''.
The new results here on the configuration also form a 
central part of this paper. We present the detailed structure 
of the intersections of corresponding edges of the two 
simplexes in perspective. Those points all lie in a hyperplane~$H$.
The consequence of the assumption that the simplexes do not 
share a face plays a crucial part in our simple proof of 
the extended Desargues theorem in Section~\ref{section:DesInNDim}.
For $n=3$ that configuration in the plane~$H$ consists of the 4 points 
and 6 lines of a complete quadrilateral. We begin with two simplexes, 
with no point or face in common, and the vertex of 
perspective,
9 points in all.  We then complete the configuration 
by adjoining the set of six intersection points in~$H$ which are disjoint 
from the simplexes and the vertex for a total of fifteen points. 
Mirabile dictu, each and every one of these 15 points then 
serves as the vertex of two simplexes in perspective with 
no point or face in common, such that the intersections 
of their corresponding edges are in the 15-point set and 
form a complete quadrilateral lying in a plane.  
The analogous result holds in all dimensions and this is 
just the beginning of the fun.  There is much more to be had!

As in~\cite{BruenBruenMcQuillan} over finite fields 
we can also, in principle, enumerate the total number
of configurations of simplexes in perspective in $n$ dimensions 
since arcs can be enumerated. The use of arcs clarifies 
the classical idea of ``points in general position''. 
Because of the method used in labelling points we know 
that points in the configuration are distinct. There is 
no concern about unexpected possible coincidences of points 
or collinearities of triples such as those detailed in 
Pedoe in the proof of the Hessenberg theorem.

The inherent combinatorial reciprocity in the configurations 
can frequently be realized as a geometrical polarity, 
at least when the characteristic is not~2.
Working over general fields, and not just the real or 
complex numbers, adds fresh insights and opens up new problems.
For example the question of the number of self-conjugate
points also arises. In the plane there are at most 4 such 
points: this is achieved only in 
characteristic~3 \cite{BruenBruenMcQuillan},
\cite{bruenMcQGeomConfigs},
\cite{bruenMcQFourSC}.
The issues above will be discussed in a future paper.

%
%
%
\section{Preliminaries: Definitions and Notation.}
\label{section:preliminaries}
a. {\em Dimension.}
In this paper the dimension of a projective space
is the projective dimension which is one less than the rank 
of the underlying vector space.  For example a projective plane 
has dimension~2 while the underlying vector space has rank~3.
The space generated by subspaces $A, B$ - their 
union or join - or by a set of points $S$ is 
denoted by $\langle A, B\rangle$ or $\langle S\rangle$ respectively. 
The rank formula for vector subspaces states the 
following: 
$rank(E+F)= rank(E) + rank (F) - rank (E\cap F)$.
For projective subspaces, with $Dim$ donating dimension, 
we have a similar dimension formula as follows:
$Dim \langle E, F\rangle  =Dim E +Dim F - Dim (E\cap F)$, where, 
$\cap$
denotes the intersection of spaces $E,F$.

[A word of caution: if the intersection is the zero vector 
its vector rank is zero. To make the formula work in the 
projective case, the dimension of the empty intersection is 
one less than zero, i.e., -1! . A test case is afforded by 
set of 2 skew lines in 3 dimensions].

b. {\em Simplex.}
In $\Sigma_n$, the projective space of dimension $n$, a 
simplex is a set $X$ of $n+1$ points which generate the space. 
This implies that each $t$-subset of $X$ generates a $(t-1)$-space.
Dually, we can also describe a simplex as a set of $n+1$ 
hyperplanes or faces, that is, subspaces of 
dimension $n-1$ generated by subsets of $X$ of size $n$.

In the plane a simplex is a set of 3 non-collinear points. 
The faces are the 3 lines joining pairs of points. 
The 3 points or 3 lines are different descriptions of 
what is really the same object ie a triangle with 
its points and lines. A simplex is a self-dual concept.
The line joining a pair of points of the simplex is called an edge. 
Each subset of $t$ points of $X$ generates a subspace of 
dimension $t-1$ for $t$ lying between 2 and $n$. 
Dually, the intersection of m faces is a subspace of 
dimension $n-1-[m-1]$, i.e., of dimension $n-m$, $m=1,2,\ldots$.

In dimension 3 a simplex is a set $X$ of 4 points, 
$X=\{1,2,3,4\}$, not all lying in a plane. 
The four faces are the 4 triangles generated by the 
sets $\{1,2,3\},\{1,2,4\},\{1,3,4\}, \{2,3,4\}$.
The simplex can be visualized by drawing a figure 
where the point 4 is above the plane generated by $1,2,3$.  
It represents a tetrahedron and, as mentioned, is 
self-dual - \cite{maxwell}.
A reciprocity can be set up by mapping each point to a face. 
For example we can map 1 to the opposite 
face $\langle234\rangle$, 2 to $\langle 134\rangle$ 3 
to $\langle 124\rangle$ and 4 to $\langle 123\rangle$. 
Then $\langle1,2\rangle$ maps to the intersection 
of $\langle 234\rangle$ with $\langle1,3,4\rangle$, 
i.e., to $\langle 34\rangle$, and so on.

%
%
%
\section{The Desargues theorem in $n$ dimensions.}
\label{section:DesInNDim}

We start with two simplexes 
$$A= \{A_1,A_2,\ldots,A_n,A_{n+1}\}\hbox{ and }
B= \{B_1, B_2,\ldots,\allowbreak B_{n+1}\}$$
in $\Sigma_n$, an $n$-space.
$A, B$ are perspective from the vertex $V$ if there are $n+1$ lines 
on $V$ with each line containing the points 
$A_i, B_i$ for $i=1,2,\ldots,n, n+1$.

\begin{theorem}
Let $A,B$ denote two simplexes in $\Sigma_n$ which 
do not share a point or face. Assume there is a 
correspondence between the points $A_i, B_i$ such 
that the edge joining $A_i$ to $A_j$ intersects 
the edge joining $B_i$ to $B_j$ in a point, $1\leq i,j\le n+1$. 
Then $A, B$ are in perspective from a vertex $V$.
  \label{theorem:simplexesInPersp}
\end{theorem}
\begin{proof}
A priori, a pair of corresponding edges might be the same 
line or they might be skew. This is ruled out by the hypothesis. 
The result holds for $n=2$ by the planar Desargues theorem.
We assume that $n\ge 3$.  
By assumption, the edges $A_1A_2$ and $B_1B_2$ are 
distinct and meet in a point. Thus the lines 
$A_1B_1$ and $A_2B_2$ meet in a point $V$. 
Let $A_3,B_3$ be any remaining pair of corresponding points.
We claim that triangles $A_1A_2A_3$ and $B_1B_2B_3$ are not coplanar. 
For suppose they lie in plane $\pi$. Let $C_4$ be the point 
of intersection of lines $A_3A_4$ and $B_3B_4$.  
If $C_4$ is in $\pi$ then $A_4$ is in $\pi$ [as is $B_4$].  
This contradicts the simplex property.
Thus the sets of points $\{A_1,A_2,A_3,A_4\}$ and 
$\{B_1,B_2,B_3,B_4\}$ both lie in the same 3-dimensional 
space generated by $A_1,A_2,A_3,C_4$.

Iterating this we end up with two corresponding faces 
namely 
$$\{A_1,A_2,\ldots,A_n\}\hbox{ and }
\{B_1,B_2,\ldots,B_n\}$$ 
lying in the same $(n-1)$-space. 
Then $A,B$ have a common face. But this contradicts the hypothesis. 
Thus the triangles are not coplanar and lie in two different planes. 

Since the line $A_i A_j$ intersects the line $B_iB_j$, 
for $j=1,2,3$ the triangles are perspective from the line of 
intersection of the two planes. Thus the triangles 
$A_1A_2A_3$ and $B_1B_2B_3$ are in perspective from 
a point~\cite[2.31]{coxeter}.  
Therefore $V$ lies on $A_3B_3$.
But $A_3$, $B_3$ are an arbitrary pair of corresponding points. 
Thus all lines $A_i B_i$ pass through $V$, for $i=1,2,\ldots,n,n+1$.
\end{proof}

\begin{lemmaN}
The point $P= P_{ij}=A_iA_j \cap B_iB_j$ is distinct 
from the point $P_{rs}$, defined similarly 
with $r,s$ replacing $i,j$ unless $\{i,j\}= \{r,s\}$.
  \label{lemma:pijDistinctFromPrs}
\end{lemmaN}
\begin{proof}
If $i=r$ and $j$ is not equal to $s$ then the points 
$A_i,A_j,A_s$ are collinear contradicting the fact 
that they are part of a simplex.
If $\{i,j,r,s\}$ consists of 4 distinct numbers 
this means that the 4 points $A_i, A_j, A_r,A_s$ all 
lie in the plane containing $P$ and the 2 lines $A_i A_j$ 
and $A_r A_s$. But this contradicts the fact that the 4 points 
form a simplex if $n=3$ or partial simplex for $n=4, 5,\ldots$.
\end{proof}

In what follows $\alpha_k$, $\beta_k$ are corresponding faces of $A,B$.

\begin{theorem}
  \label{theorem:simplexesIntersection}
Let $A,B$ denote simplexes in the $n$-dimensional space $\Sigma_n$ as in 
Theorem~\ref{theorem:simplexesInPersp}.
Then
\begin{enumerate}[a.]
  \item
  \label{theorem:simplexesIntersectionParta}
The intersections of corresponding $t$-spaces are $(t-1)$-spaces,
for $t=1,2,\ldots,n-1$.
  \item
  \label{theorem:simplexesIntersectionPartb}
The spaces $(\alpha_i\cap \alpha_j)$ with $(\beta_i \cap \beta_j)$
generate an $(n-1)$ space, a hyperplane of 
$\Sigma_n$,
where $i,j$ are distinct.
\end{enumerate}
\end{theorem}
\begin{proof}[Proof of a.]
The result holds, by hypothesis, for $t=1$.
Let $t=2$. We have 2 triangles $A_1,A_2,A_3$ and $B_1,B_2,B_3$.
Each edge of the triangle from $A$ meets the corresponding edge 
from $B$ in a point. The line of intersection of the planes of 
the triangles is a line containing 3 distinct 
points $P_{ij}$, $1\le i,j \le 3$.

Let $t=3$.  We have two tetrahedra $A,B$ with 4 pairs of corresponding faces. 
Each face contains a triangle. Each corresponding pair of faces 
intersect in a line. 
From Lemma~\ref{lemma:pijDistinctFromPrs}
no two of the lines are the same. 
Thus, for $t=3$, the dimension of the 
intersection is at least two.
To establish the result we induct on $t$.  
Assume, by induction, that the intersection of two $t$-spaces 
has dimension $t-1$.  The intersection of two 
corresponding $(t+1)$-spaces contains the union of 
the intersections of corresponding pairs of $t$-spaces. 
Each such pair intersect in a $(t-1)$-space. 
Distinct pairs yield distinct intersections from 
Lemma~\ref{lemma:pijDistinctFromPrs}.
Thus the dimension of the intersection is at least~$t$. 
From the dimension formula the corresponding union 
has dimension at most $t+1$.  
From the argument in the proof of
Theorem~\ref{theorem:simplexesInPersp}
the dimension of that union of the two subspaces is at least $t+1$. 
We conclude that the dimension of the union is $t+1$, 
and that of the intersection is $t-1$, for all 
values of $t$ between~1 and $n-1$.
\end{proof}
\begin{proof}[Proof of b.]
$Dim \langle (\alpha_i\cap \alpha_j)\rangle$ 
and $Dim \langle(\beta_i\cap \beta_j) \rangle$ 
are $n-2$ as each face is generated by $n-1$ points of $A,B$.
From part~a the dimension of their intersection is $n-3$.  
From the dimension formula the dimension of their union is $n-1$.
\end{proof}

\begin{theorem}
  \label{theorem:twoSimplexesFourParts}
Let $A,B$ denote two simplexes in perspective from a 
vertex $V$ in $\Sigma_n$ which have no common points or faces. 
Then 
\begin{enumerate}[a.]
\item
The intersection points of corresponding edges form a set $S$ of 
$n+1 \choose 2$ 
distinct points in $\Sigma_n$. 
None of these points lie in~$A$ or~$B$ or are equal to~$V$.
\item
The intersection of corresponding $t$-spaces of $A,B$ 
is a $(t-1)$-space in $\Sigma_n$ for $t=1,2,\ldots,n-1$.
\item
The intersection of two corresponding faces of $A,B$ 
is an $(n-2)$-space which lies in a fixed hyperplane $v$ of $\Sigma_n$.
\item
In particular the set $S$ lies in $v$.
\end{enumerate}
\end{theorem}
\begin{proof}
 Since $A,B$ have no common points the lines $A_iB_i, A_jB_j$, 
for distinct $i,j$, exist, are distinct and contain $V$. 
It follows that lines $A_i A_j$ and $B_i B_j$ intersect in a 
unique point which cannot be $V$ and
which does not lie in~$A$ or~$B$. From 
Lemma~\ref{lemma:pijDistinctFromPrs},
$S$ contains ${n+1 \choose 2}$ distinct points. 
This proves part~a.

Since corresponding edges meet in a point,
part~b follows from 
Theorem~\ref{theorem:simplexesIntersection}
part~\ref{theorem:simplexesIntersectionParta}

The dual of 
Theorem~\ref{theorem:simplexesInPersp}
asserts that the intersection of pairs of corresponding 
faces of $A, B$ lies in a fixed hyperplane $v$ provided that 
Theorem~\ref{theorem:simplexesIntersection},
part~\ref{theorem:simplexesIntersectionPartb},
holds. 
Since 
Theorem~\ref{theorem:simplexesIntersection}
applies, in particular, to simplexes $A, B$ in perspective 
from a point this proves 
part~c. 
Each point of $S$ lies in the intersection of (several pairs of) 
corresponding faces of $A, B$. This proves part~d.
\end{proof}

We have now shown a (strong) converse to 
Theorem~\ref{theorem:twoSimplexesFourParts},
as follows.
\begin{theorem}
Let $A,B$ be simplexes which do not share a common point or face. 
Assume that there is a correspondence between points $A_i, B_i$ 
such that lines $A_i A_j$, $B_i B_j$ meet in a point. 
Then $A,B$ are perspective from a point and a, b, c, d of 
Theorem~\ref{theorem:twoSimplexesFourParts},
hold.
\end{theorem}

%
%
%
\section{Arcs and coordinate systems}
\label{section:ArcsAndCoordSystems}

In an $n$-dimensional projective space $\Sigma_n$ a 
{\em coordinate system} is a set of $n+2$ points such that any 
subset of size $n+1$ is a simplex~\cite{hirschfeld}.
A simplex and a coordinate system are examples of arcs.
An {\em arc} in $\Sigma_n$ is a set of $m$ points, with $m$ at least $n+1$, 
with the property that every subset of size $n+1$ forms a simplex. 
This implies that any subset of the arc of size $t$ generates a 
subspace of dimension $t-1$ when $t$ is at most $n+1$.  
If the underlying field is finite, of order $q$, it is conjectured that, 
if $n \le q$, the maximum size of an arc in $\Sigma_n$ 
is $q+1$ if $q$ is odd $q$ or $q+2$ if $q$ even: 
see~\cite{BruenThasBlokhuis}.

\begin{theorem}
For each $n$ there exists a coordinate system $\Gamma_{n+2}$ 
in $\Sigma_n$ such that no point lies in a given hyperplane $H$ 
of $\Sigma_n$ so long as the underlying field $F$ has order greater than~2.
\end{theorem}
\begin{proof}
Using homogeneous coordinates $(x_1,x_2,\ldots x_n, x_{n+1})$ let $K$
be the hyperplane with equation $x_1+x_2+x_3+\cdots+x_n+ x_{n+1}=0$. 
We choose the first $n+1$ points $P_i$ of the arc $\Gamma_{n+2}$ to 
have 1 in position $i$ and zeros elsewhere, $i=1,2,3,\ldots,n, n+1$.  
These points form a simplex. Point number $n+2$ is $(1,1,1,\ldots,1,z)$,
 where $z$ is any non-zero number in $F$. This point 
is off $K$ provided $n +1+z$ is non- zero. This can always be done 
if $F$ has order greater than~2.  These $n+2$ points form a coordinate system 
with no point in $K$.
Finally we use a collineation of $\Sigma_n$ mapping $K$ to the 
given hyperplane $H$. This will map the set of $n+2$ points 
above to a coordinate system having none of its points on $H$.
\end{proof}

\begin{remarkN}
If $F$ has order~2 then 
Theorem~\ref{theorem:simplexesInPersp}
is false. To see this let $n=3$. Let $W$ be a 5-arc in 
$\Sigma_3=PG(3,2)$. Suppose the result holds. 
Then as in~\cite{BruenBruenMcQuillan}
the section by $PG(2,2)$ yields a configuration with 10 points. 
But $PG(2,2)$ only has 7 points.
\end{remarkN}

Henceforth we assume that $F$ has order greater than~2.

%
%
%
\section{The section of a coordinate system in $\Sigma_{n+1}$ by a 
hyperplane $H=\Sigma_n$.}
\label{section:secCoordSysByHyperplane}

We start with $\Sigma_n$ embedded as a hyperplane $H$ in $\Sigma_{n+1}$. 
Let $\Gamma=\Gamma_{n+3}$ denote a coordinate system in 
$\Sigma_{n+1}$ consisting of $n+3$ points $1,2,\ldots,n+2, n+3$ 
where none of these points lie in the hyperplane $H$. 
By joining pairs of points of $\Gamma$ we 
generate ${n+3 \choose 2}$ lines in $\Sigma_{n+1}$. 
The section of the line joining points $i,j$ of the 
arc by $H$ is the point denoted by the unordered pair $(i,j)$. 
All told this yields ${n+3\choose 2}$ points in $H$, with $i,j$ 
lying between 1 and $n+3$.

We claim that these points are distinct. From the arc property 
no 3 points of $\Gamma$ are collinear. 
Further, suppose that the 
line joining~1 to~2 meets the line joining~3 to~4 say in 
the same point of~$H$. Then the 4 points $1,2,3,4$ lie in a plane, 
which is forbidden by the arc property.

A triangle formed from 3 points of $\Gamma$, say $\{1,2,3\}$ 
has as its section by $H$ three collinear points $(1,2),(1,3),(2,3)$. 
In general, points $(i,j)$ and $(k,l)$ lie on a line in $H$ 
if and only if the points have a symbol in common, in which case 
they lie on a line with 3 points.

\begin{exampleN}
\label{example:completeQuadrilateral}
Suppose we section the figure formed from 4 of the 6 points in~$\Gamma$, 
say the points in $S=\{1,2,3,4\}$, by $H$. 
These points form a tetrahedron with 4 points, 
6 lines and 4 planes in $\Sigma_3$.
The section of the figure has 
${4 \choose 2}$, i.e, 
6 points and 4 lines. These lines come from the 4 triangles 
formed by the 4 triples in $S$. 
Any pair of triples share a pair of points.
Thus any 2 of these 4 lines in $H$ meet in a point, 
and no 3 lines are concurrent. 
This figure lies in a plane in H and is a 
complete quadrilateral \cite[p.\ 7]{coxeter}, \cite[p.\ 7]{lord}
If we project the tetrahedron to $H$ from a general point of 
$\Sigma_3$ we end up with a planar figure having, 
dually, 6 lines and 4 points which is known as a 
complete quadrangle \cite[p.\ 7]{coxeter}, \cite[p.\ 18]{lord}.

If we project just the points in the set $\{2,3,4\}$ 
from the point~1 to $H$ we obtain a triangle with vertices 
$(1,2),(1,3),(1,4)$.
From these 3 points we generate 3 more,
namely $(2,3),(2,4),(3,4)$.
Thus we end up with the same set of six points as the section of 
$\{1,2,3,4\}$ by $H$.
\end{exampleN}

\begin{theorem}
  \label{theorem:t+1arcGamman+3} 
Let $S$ denote a subset of size $t+1$ of an 
arc $\Gamma_{n+3}$ in $\Sigma_{n+1}$.  Then
\begin{enumerate}[a.]
  \item
  $S$ generates a subspace of dimension $t$ in the space 
  $\Sigma_{n+1}$.
  \item
  The section of $\langle S\rangle$ by the hyperplane 
  $H=\Sigma_n$ generates a subspace $M$ of dimension $t-1$ in $H$. 
  $M$ is also generated by the projection of $S$ to $H$ from 
  any point of $S$.
\end{enumerate}
In summary the section of (the space generated by) a set $S$ 
of $t+1$ points of the $(n+3)$-arc in 
$\Sigma_{n+1}$ by $H$ is a space of dimension $t-1$ in 
$H=\Sigma_n$.
\end{theorem}
\begin{proof}
Part~a follows from the definition of an arc.

For Part~b, let 
$S=\{1,2,3,\ldots,t, t+1\}$.
Denote by $L$ the subspace generated by 
$\{2,3,\ldots,t, t+1\}$.
$L$ has dimension $t-1$.
Projecting
$L$ from the point~1,
which is not in $L$, we obtain a $(t-1)$-dimensional space $M$ 
generated by the $t$ points $(1,2),(1,3),\ldots,(1,t), (1,t+1)$ 
in $\Sigma_n$. Since $(1,2),(1,3)$ are in $M$ 
the point $(2,3)$ on the line joining them is also in $M$. 
Then the points $(2,1)$ [$=(1,2)$], $(2,3), (2,4),\ldots,(2,t), (2,t+1)$ 
are in $M$. Proceeding, we see that $M$ contains all points 
$(i,j)$ where $i, j$ lie between 1 and $t+1$ so $M$ 
is generated by the section of $\langle S\rangle$. 

We generate the same space $M$ by using any point in $S$ 
instead of the point~1.

Alternative proof of~b. 
As shown in 
Example~\ref{example:completeQuadrilateral}
for the case $n=3$,
$M$ is generated by
$\{(1,2), \ldots, (1,t), (1,t+1)\}$, and has 
dimension $x$ say. 
Since the point~1 is not in 
$\Sigma_n$ it is not in $M$. Thus 
$K=\langle 1,M\rangle$ has dimension $x+1$.
Since $K$ contains points 1 and $(1,j)$ it contains 
the line joining them and so the point $j$ for 
$j=2,3,\ldots,t, t+1$. $K$ also has the point 1.
Thus $K$ contains $t+1$ arc points and has dimension $t$. 
Therefore $t=x+1$, so $x=t-1$. 
Thus $M$, the section of the $(t+1)$-set $S$ of the arc 
by $H$ generates a space of dimension $t-1$.
\end{proof}

From the above the section of the space generated by a 
$(t+1)$-set of the arc $\Gamma$ by $H$ is a 
$(t-1)$-space.
We also have the following result.

\begin{theorem}
Given the arc $\Gamma$ with $n+3$ points in 
$\Sigma_{n+1}$ the section by the hyperplane 
$H=\Sigma_n$ has ${n+3\choose 2}$ points, ${n+3\choose 3}$ lines,
${n+3\choose 4}$ planes, $\ldots$, and 
${n+3\choose n+1}$ spaces of dimension $n-1$
i.e. hyperplanes in $H=\Sigma_n$.
\end{theorem}

%
%
%
\section{From arcs to simplexes in perspective.}
\label{section:arcsToSimplexesInPersp}

We continue with the same notation. 
$H=\Sigma_n$ is a hyperplane in 
$\Sigma_{n+1}$. 
$\Gamma=\Gamma_{n+3} = \{1,2,3,\ldots,n+3\}$ 
is an arc of $n+3$ points in the space $\Sigma_{n+1}$.

\begin{theorem}
  \label{theorem:n+1PointsSimplex}
Let $S=\{(1,3), (1,4),\ldots,(1,n+2), (1,n+3)\}$.
Then the $n+1$ points of $S$ 
generate a subspace of dimension $n$
and form a simplex in~$H$.
\end{theorem}
\begin{proof}
This follows as in the proof of
Theorem~\ref{theorem:t+1arcGamman+3}.
\end{proof}

Similarly we have:
\begin{theorem}
  \label{theorem:similarlyPointsSimplex}
  Let $T= \{(2,3), (2,4),\ldots,(2,n+2),(2,n+3)\}$. 
  Then the points in $T$ form a simplex
  in $H$.
\end{theorem}

\begin{theorem}
  \label{theorem:STAugmentedBy12Arc}
  The sets $S,T$, augmented by the point $(1,2)$, 
  form a coordinate system yielding an arc of 
  size $n+2$ in $H$ in each case.
\end{theorem}
\begin{proof}
  $S$ is a simplex. We show that any $n$-subset $Z$ of $S$, 
  when augmented by $(1,2)$ is a simplex. 
  Let $Z=\{(1,3),(1,4),\ldots,(1,n+2)\}$. 
  We must show that $X=\{(1,2),(1,3),\ldots,(1,n+2)\}$
  is a simplex. $X$ is the projection of 
  $U=\{2,3,4,\ldots,n+2\}$ from the point 1 onto $H$. 
  Since $\langle U\rangle$ has dimension $n$ 
  we conclude that $X$ is a set of $n+1$ points 
  that generates a space of dimension $n$ 
  i.e. $X$ is a simplex in $H$.

  Similarly $T$, augmented by $(1,2)$, forms a coordinate system.
\end{proof}

\begin{theorem}
  \label{theorem:simplexesPointsFaces}
  \begin{enumerate}[a.]
    \item
    The sets $S,T$ are simplexes which are in perspective from 
    the vertex $V=(1,2)$.
    \item
    $S,T$ have no points in common.
    \item
    $S$ and $T$ have no faces in common.
  \end{enumerate}
\end{theorem}
\begin{proof}
  From 
  Theorem~\ref{theorem:n+1PointsSimplex}
  and 
  Theorem~\ref{theorem:similarlyPointsSimplex},
  $S$ and $T$ are simplexes. 
  They are in perspective from the point $V=(1,2)$ because $(1,2)$ 
  is collinear with points $(1,i)$ and $(2,i)$ for $i$ 
  between 3 and $n+3$.
  The labels of their points show that $S$ and $T$ have no points in common.

  For part~c, suppose that $S$ and $T$ have a face 
  $\Lambda$ in common. 
  The lines joining the $n$ corresponding points 
  $A_i, B_i$ in $\Lambda$ all contain 
  the vertex $V$. 
  Then $\Lambda$, augmented by $V$, is not a
  simplex in $H$, contradicting
  Theorem~\ref{theorem:STAugmentedBy12Arc}.
\end{proof}

%
%
%
\section{From simplexes in perspective to arcs.}
\label{section:simplexesToArcs}

Our goal now is to show that two simplexes in perspective 
which have no common points or faces arise from an arc as 
developed in 
Section~\ref{section:arcsToSimplexesInPersp}.

\begin{theorem}
\label{theorem:twoSimplexesNoPtNoFaceArcHyperplane}
In $\Sigma_n$ let 
$$A =\{A_3,A_4,\ldots,A_{n+2}, A_{n+3}\}$$
and 
$$B=\{B_3,B_4,\ldots,B_{n+2}, B_{n+3}\}$$
denote two simplexes in $\Sigma_n$ 
which are in perspective from a point $V$
and have no point in common.
Assume also that the two simplexes have no face 
in common or, equivalently, that 
$V$ does not lie on a face of $A$ or of $B$.
Then, as in 
Theorem~\ref{theorem:simplexesPointsFaces},
$A,B$ arise from the section of the 
space generated from an arc of size $n+3$ in 
$\Sigma_{n+1}$ by $H=\Sigma_n$, an $n$-dimensional 
hyperplane of $\Sigma_{n+1}$.
\end{theorem}
\begin{proof}
We choose a line $l$ on $V$ not lying in 
$H$ and on it we choose two points 
labelled $1,2$ which are different from $V$. 
Define a set $\Gamma_{n+3}$ as follows:
Point~$i$ is the intersection of the lines $1A_i$ and $2B_i$, 
$3\le i\le n+3$.
This is well-defined as, for each $i$, those two lines 
lie in a plane containing two distinct lines on $V$ i.e. 
the lines joining $V$ to the point~$1$ and the point $A_i$.
We claim that $\Gamma_{n+3} = \{1,2 ,3, 4,\ldots,n+2, n+3\}$ 
is an arc of size $(n+1) +2$ yielding a coordinate system 
in $\Sigma_{n+1}$ such that none of its points lie in $\Sigma_n$.

We must show that every $(n+2)$-subset of $\Gamma_{n+3}$ 
generates the space $\Sigma_{n+1}$.
Such a subset must contain either 1, or 2, or both.

Case 1.
We show that the $n+2$-set 
$S= \{1,3,4,\ldots,n+2, n+3\}$ 
generates $\Sigma_{n+1}$ as follows. 
The $n+1$ points $\{A_3,A_4,\ldots,An+3\}$ 
generate a face of $A$, which is an $n$-space in 
$\Sigma_n$ as $A$ is a simplex. 
The face is a projection of the space 
$U=\langle \{3, 4,\ldots,n+2,n+3\}\rangle$
from the external point~1.
Thus $U$ has dimension $n$. 
When we adjoin~1, the enlarged space, generated by 
$\{1,3, 4,\ldots,n+2, n+3\}$, has dimension~$n+1$.
[Alternatively the above face of $A$ is an $n$-space. 
Since the point~$1$ is not in $\Sigma_n$ 
adjoining it to this face generates an
$n+1$-space, namely the space generated by
$\{1,3,4,\ldots,n+2,n+3\}$].

Case 2. 
The $n+2$-set $\{2,3,4,\ldots,n+2,n+3\}$ 
also generates $\Sigma_{n+1}$.
The proof is the same as for 
Case~1 upon interchanging points $1, 2$.

Case 3.
We show that the set $\{1,2,3,\ldots,n+2\}$ 
generates $\Sigma_{n+1}$. 
The $n$ points 
$\{A_3, A_4,\ldots,A_{n+2}\}$ 
form a face of $A$ and generate a hyperplane $K$ of 
dimension $n-1$ in 
$\Sigma_n$. 
By hypothesis, the point $V$ is not in $K$, so 
$\langle K,V\rangle$ has dimension $n$. 
Adjoining the point~$1$ to $\langle K,V\rangle$ 
yields an $(n+1)$-dimensional space $W$ in
$\Sigma_{n+1}$.
The points $1,2$ and $V$ are collinear. 
So $W$
contains $1$ and $2$.
Since the points $1,i$ and $A_i$ are collinear,
$W$ also contains points $3,4,\ldots,n+1,n+2$.
Since $\langle 1,2,3,\ldots,n+1,n+2\rangle$ 
contains $K,V$ and the point~$1$ it is an $(n+1)$-space 
contained in $W$ so it must be $W$.
In summary, the set 
$\Gamma_{n+3}$ above is an $(n+3)$-arc in 
$\Sigma_{n+1}$. 
Moreover, no point of it lies in $\Sigma_n$.
To see this, the points $1,2$ lie outside $\Sigma_n$. 
Suppose that a point~$i$ other than the points $1,2$ 
lies in $\Sigma_n$. 
Since $A_i$ lies in $\Sigma_n$ and the points 
$1,i, A_i$ are collinear, this implies that the 
point~$1$ is in $\Sigma_n$, which is a contradiction.

We examine the section of the arc 
$\Gamma_{n+3}$ in 
$\Sigma_{n+1}$ by the hyperplane $H=\Sigma_n$, 
assigning new labels to the two simplexes, and to $V$, as follows.
$V$ is on the line joining points $1$ and $2$ and is relabelled $(1,2)$.
$A_i$ is on the line joining points $1$ and $i$ so it becomes $(1,i)$.
Similarly $B_i$ is now the point $(2,i)$ for 
$i= 3,4,\ldots,n+2$.

Because the points $(1,i),(2,i)$ and $(1,2)$ 
are collinear we have that $A_i$ and $B_i$
are in perspective from the vertex $(1,2)$.
In summary the two simplexes $A, B$ 
are contained in the section of the arc 
$\Gamma_{n+3}$ in $\Sigma_{n+1}$ 
which yields a set $X$ of 
${n+3\choose 2}$ 
points in $\Sigma_n$. 
$X$ contains two simplexes, 
$A=\{(1,3), (1,4),\ldots,(1,n+3)\}$, $B=\{(2,3), (2,4),\ldots,(2,n+3)\}$.
They are in perspective from $(1,2)$ and they do not 
share a point. The vertex of perspective 
which is the point $(1,2)$, does not lie on
any face of $A$ or $B$.
Thus $A$ and $B$ do not share a face.
This proves 
Theorem~\ref{theorem:twoSimplexesNoPtNoFaceArcHyperplane}.
\end{proof}

%
%
%
\section{An extension of the Desargues theorem.}
\label{section:extensionOfDesargues}

It is time to reap the benefit of the work in 
Sections~\ref{section:ArcsAndCoordSystems}-\ref{section:simplexesToArcs}.
This section contains a second proof of the extended Desargues theorem 
in all dimensions.

\begin{definitionN}
  Let $A, B$ be simplexes of $\Sigma_n$ which have 
  no common points or faces. They are defined to be
  in perspective from a hyperplane $v$ if there is 
  a correspondence between the points of $A,B$ – and 
  therefore the subspaces of $A,B$ - such that the following holds:
  the intersection of corresponding $t$-spaces of $A,B$ 
  is a $(t-1)$-space lying in a fixed hyperplane $v$ for 
  $t=1,2,\ldots,n-1$.
\end{definitionN}

\begin{theorem}
  \label{twoSimplexesNoPointNoHyperplaneTwoInPerspective}
  Let $A, B$ denote two simplexes in the space 
  $\Sigma_n$ with no common point or hyperplane.
  Then, if $A, B$ are in perspective from a point 
  they are in perspective from a hyperplane.
\end{theorem}
\begin{proof}
  From 
  Theorem~\ref{theorem:twoSimplexesNoPtNoFaceArcHyperplane}.
  we may assume that $A,B$ are as follows:
  $$A= \{(1,3), (1,4),\ldots,(1,n+3)\},$$
  $$B=\{(2,3), (2,4),\ldots,(2,n+3)\}.$$
  They are in perspective from the vertex $(1,2)$ 
  since the line joining corresponding points $(1,i)$ and $(2,i)$
  contains the vertex $(1,2)$,
  $3\le i\le n+3$.
  The edge of $A$ joining $(1,i)$ to $(1,j)$ contains the point $(i,j)$.
  The edge of $B$ joining $(2,i)$ to $(2,j)$ contains the point $(i,j)$.
  Moreover each point $(i,j)$ arises as the intersection of 
  two corresponding edges,
  $3\le i,j\le n+3$.
  Thus the set $S$ of points $(i,j)$, which is 
  the set of all intersections of corresponding edges of $A$ with $B$ 
  is found as the section of the configuration generated 
  by the $(n+1)$-set $U=\{3,4,\ldots,n+3\}$ in $\Sigma_{n+1}$ 
  by the hyperplane $H=\Sigma_n$. 
  $U$ is an $(n+1)$-subset of the arc $\Gamma_{n+3}$ in $\Sigma_{n+1}$.

From 
Theorem~\ref{theorem:t+1arcGamman+3} part~b, 
$Dim(\langle S\rangle)= n-1$ and $S$ generates a hyperplane 
in $H=\Sigma_n$. 
The number of points in $S$ is ${n+1\choose 2}$.

More generally, let $Q,R$ denote sets of $t+1$ 
corresponding pairs of points of $A,B$. 
For example let $Q=\{(1,3) ,(1,4),\ldots,(1,t+3)\}$ and
and let $R=\{(2,3),(2,4),\ldots,(2,t+3)\}$.
Then $Q,R$ each generate a $t$-space. 
As above, in the case when $t=n$, the set of intersections of 
all pairs of corresponding edges consists of points $(i,j)$ 
for $i,j$ lying between $3$ and $t+3$ with $i$ unequal $j$. 
These points generate a space of dimension $t-1$ in $H$.
\end{proof}

The dual of 
Theorem~\ref{twoSimplexesNoPointNoHyperplaneTwoInPerspective},
where we use the reciprocity that interchanges points
and faces as mentioned in Section~\ref{section:preliminaries},
provides a converse to it as follows.

\begin{theorem}
  Let $A, B$ be two simplexes in a 
  projective space $\Sigma_n$ with no common point 
  and no common hyperplane. Then if $A,B$ are in perspective 
  from a hyperplane they are in perspective from a point.
\end{theorem}

%
%
%
\section{The Configurations.}
\label{section:theConfigs}

To recap, we are working in $\Sigma_n$ which is contained 
as a hyperplane $H$ in $\Sigma_{n+1}$.
The arc $\Gamma_{n+3}$ in $\Sigma_{n+1}$, 
when sectioned by $H$, yields a set $W$ of 
${n+3\choose 2}$ points
in $H$. As in 
Section~\ref{section:arcsToSimplexesInPersp},
$W$ contains two simplexes $A, B$ accounting for $2(n+1)$ points. 
The vertex of perspective is one point. Then, from 
Section~\ref{section:extensionOfDesargues},
the intersections of pairs of corresponding edges yield 
${n+1\choose 2}$ 
additional points in a hyperplane of
$H=\Sigma_n$.
This accounts for all points of the configuration! 
The proof follows from the following identity.

\begin{equation}
  {n+3 \choose  2} =2(n+1) +1 +{n+1 \choose 2}.
\end{equation}
\bigskip

From the manner of assigning pairs $(i,j)$ to points in $H$, 
different pairs yield different points. Thus the points of the 
simplexes, the vertex and the edge-intersections 
are distinct sets of points.

\begin{theorem}
  \label{theorem:n+3choose2pointsVertexPersp}
  Each of the ${n+3\choose 2}$ points 
of the configuration is a vertex of perspectivity for a 
pair of simplexes having no points or faces in common.
\end{theorem}
\begin{proof}[Proof 1.]
Choose any point $(i,j)$ in $H$. 
It lies on the line joining points $i$ and $j$ of the arc. 
It will be the vertex. The section by $H$ of the lines 
joining $i,j$ to the other points of the existing arc 
yields two simplexes in perspective from the point $(i,j)$.
The intersections of corresponding edges of the new pair 
of simplexes also lie in the original configuration.
\end{proof}
\begin{proof}[Proof 2.]
Any permutation $T$ of $\{1,2,3,....n+3\}$ when applied to
the points of the configuration yields two simplexes in 
perspective from a vertex, with the intersections 
of corresponding edges lying in a hyperplane. 
This is so because two points $(x,y)$ and $(z,w)$
are contained in a line i.e., share a symbol, 
if and only their images under $T$ share a symbol.
\end{proof}

\begin{unnamed}
  \label{unnamed:initialArcInSigma5}
The case $n=4$. 
The underlying initial arc in $\Sigma_5$
has $5+2$, i.e. $7$, points named $1,2,3,\ldots,7$. 
The section of lines joining pairs of points of the arc 
by the hyperplane $H=\Sigma_4$ yields a set of 21 distinct points.
From 
Theorem~\ref{twoSimplexesNoPointNoHyperplaneTwoInPerspective}
the section of the lines joining pairs of points of 
the 5-subset $X=\{3,4,5,6,7\}$ by $H$ generates a 
$\Sigma_3$ subspace of $H$ denoted by $v$. 
The section of (the lines generated by) $X$ contains 
10 points $(i,j)$ where $i$ and $j$ are different elements of $X$. 
These 10 points are the intersections of pairs of 
corresponding edges of the given simplexes $A, B$ in $H$
as in Theorem~\ref{twoSimplexesNoPointNoHyperplaneTwoInPerspective}.
\end{unnamed}

We mention some facts on the structure of these ten points. 
Choose any of the ten points to be a vertex, say the point $(3,4)$. 
Then we have two triangles namely $\{(3,5), (3,6), (3,7)\}$ and 
$\{(4,5),(4,6),(4,7)\}$ which are in perspective from $(3,4)$.
From the arc property the triangles lie in different planes.
[If, for example, $(4,7)$ was in the plane
containing $(3,5),(3,6),(3,7)$
then the arc points $3,4,5,6,7$
would only generate a 3-space.]
The intersection of the corresponding triangle edges 
lies on the line of intersection of the two planes.
The intersection points are $\{(4,5),(4,6),(4,7)\}$.
The two planes of the triangles 
in perspective meet in a line in $v$. 
Each line in $v$ is descended from a 3-subset in $\Sigma_5$. 
A 3-set such as $\{3,4,5\}$ lies in two 4-subsets of 
$\{3,4,5,6,7\}$.  Thus each line in $v$ lies in just 2 planes of $v$.
Each point such as $(3,4)$ lies on 3 lines formed from the triples 
$\{345\},\{346\}, \{347\}$ and on 3 planes formed from the 
4-sets $\{3456\} ,\{3457\},\{3467\}$.

The structure of the 10 points is symmetric in the 
sense that any point is the vertex of perspective of 
two triangles such that the intersection of pairs 
of corresponding edges lie on a line, the axis of perspectivity.

As mentioned above each of the 21 points serves as a 
vertex of perspective of two simplexes in $\Sigma_4$. 
The intersections
of corresponding edges yield a set of 10 points 
as described above. So we have 21 such sets of ten points.

%
%
%
\section{Self replication of configurations.}
\label{section:selfReplicationOfConfigs}

For want of better terminology we first 
define what is meant by a ``semi-simplex''.

\begin{unnamed}
In $\Sigma_n$ a semi-simplex is defined to be a set 
of $n$ points which generate an $(n-1)$-space.
A simplex is a set of $n+1$ points which generate the $\Sigma_n$.
We examine the configuration of a pair of semi-simplexes 
in perspective and the intersections of corresponding edges.

$n=1$.
Here a semi-simplex pair is a set of 2 points $A,B$ on a line $l$. 
Let $V$ be another point on $l$. 
Then $A,B$ are in perspective from $V$. 
So the configuration is a line with 3 points $A,B$ and $V$.

$n=2$.
In the plane a semi-simplex is a pair of points $A_1, B_1$ on a line $L_1$. 
Let $A_2$, $B_2$ form another line $L_2$ in the plane. 
Let $V$ denote the intersection of $A_1A_2$ with $B_1B_2$.
We now have two semi-simplexes in perspective from $V$.

Next, the intersections of corresponding edges is the point 
$C_3$ where lines $A_1A_2$ and $B_1B_2$ meet. 
In summary we have a complete quadrilateral with 6 points.

$n=3$.
A semi-simplex is a triangle. A pair of semi-simplexes 
in perspective is simply is a pair of triangles 
in perspective 
from a vertex. In our situation, 
because of the arc property,
the triangles will not be coplanar. The general result is as follows.
\end{unnamed}

\begin{theorem}
\label{theorem:twoSimplexesArcPointsn+3choose3}
In $\Sigma_n$,
let $A,B$ denote two simplexes in perspective from a point $V$ 
such that $A,B$ share no points or hyperplanes. Let 
$\Gamma= \Gamma_{n+3}$ denote the underlying arc in 
$\Sigma_{n+1}$ giving rise to $A,B$ as in 
Section~\ref{section:simplexesToArcs}.
Let $X_n$ denotes the ${n+3 \choose 2}$ 
points $(i,j)$ in $\Sigma_n$ which are the section of lines joining points
$i,j$ of the arc $\Gamma$.  Then
\begin{enumerate}[a.]
  \item
  $X_n$ consists of two simplexes $A, B$ in perspective, 
  the vertex of perspective, and a subset $Y_n$ of $X_n$ 
  consisting of the intersections of pairs of corresponding edges of $A,B$.
  \item
  The points of $Y_n$ form a semi-simplex pair $C,D$ 
  in perspective from a vertex in $\Sigma_{n-1}$.
  This pair, the vertex of perspective and the intersections 
  of pairs of corresponding edges 
  of $C,D$ account for all the points in $Y_n$.
  \item
  The intersections of pairs of corresponding edges of $C,D$ 
  form a semi-simplex pair $E,F$ 
  in 
  $\Sigma_{n-2}$
  which are in perspective from a vertex.
\end{enumerate}
\end{theorem}
\begin{proof}
  The general case is analogous to the configuration of 10 points in 
  $\Sigma_3$ for the case $n=4$ in~\ref{unnamed:initialArcInSigma5}.
  $X_n$ consists of the ${n+3 \choose 2}$ 
  points $(i,j)$ for $i,j$ between $1$ and $n+3$. 
  As in 
  Theorem~\ref{twoSimplexesNoPointNoHyperplaneTwoInPerspective}
  the intersections of pairs of corresponding edges is 
  the subset $Y_n$ of $X_n$ 
  consists of points $(i,j)$ with $i,j$ lying between $3$ and $n+3$.

  $Y_n$ lies in $K$, with $K$ a hyperplane of 
  $\Sigma_n$ of dimension $n-1$, and contains ${n+1 \choose 2}$ 
  points. If we choose say, the point $(3,4)$ as vertex 
  we have two min-simplexes in perspective from it, 
  namely $\{(3,5),(3,6),\ldots,(3,n+3)\}$ and 
  $\{(4,5),(4,6),\ldots,(4,n+3)\}$.

  The set $Y_n$, lying, in $\Sigma_{n-1}$, 
  contains 2 mini-simplexes in $\Sigma_{n-1}$, 
  each having $n-1$ points, 
  along with the point $(3,4)$ as a vertex of perspective
  and contains also ${n-1 \choose 2}$ 
  points of intersection of corresponding edges which lie in a 
  $\Sigma_{n-2}$.

  This accounts for all points in $Y_n$ 
  as follows from the following identity:

\begin{equation}
  {n+1 \choose 2} = 2(n-1) +1 + {n-1 \choose 2}.
\end{equation}

This identity is simply 
Theorem~\ref{twoSimplexesNoPointNoHyperplaneTwoInPerspective}
with $n$ replaced by $n-2$.
Part~c follows from an iteration of the above procedure.
\end{proof}

%
%
%
\section{Some further extensions of Desargues theorem.}
\label{section:furtherExtensionsOfDes}

Using the above notation,
we consider 3 semi-simplexes $A, B,C$ in $\Sigma_n$ 
such that each pair is in perspective from one of 
three vertices of perspective that lie on a line in $X_n$.
Recall that $X_n$ is the set of ${n+3 \choose 2}$ 
points obtained from the arc $\Gamma_{n+3}$ 
consisting of the points $1,2,\ldots,n+2,n+3$.

Without loss of generality the 3 vertices are $(1,2)$, $(1,3)$ 
and $(2,3)$. Neither one of the pair of semi-simplexes 
in perspective from $(1,2)$ is allowed to contain $(1,3)$ or $(2,3)$.
The pair $A,B$ of semi-simplexes, perspective from $(1,2)$ 
must look like $\{(1,4),(1,5),\ldots,(1,n+2) ,(1,n+3)\}$, 
$\{(2,4),(2,5),\ldots,(2,n+2), (2,n+3)\}$. 
If the simplex $C$ is in perspective with $B$ from the point $(2,3)$ then
$C=\{(3,4), (3,5),\ldots,(3,n+2), (3,n+3)\}$.
The pair $A,C$ are then
in perspective from the point $(1,3)$.

We now have the following result which is shown for the 
case $n=3$ \cite[p.\ 64]{lord}.
\begin{theorem}
  Let $A,B ,C$ be three semi-simplexes in 
$\Sigma_n$ such that each of the 3 pairs are 
in perspective from one of three collinear points. 
Then each pair of semi-simplexes is perspective 
from the same hyperplane in $\Sigma_n$.
\end{theorem}
\begin{proof}
As in the proof of 
Theorem~\ref{twoSimplexesNoPointNoHyperplaneTwoInPerspective}
the intersections of corresponding pairs of edges for each 
pair of simplexes lies in a hyperplane $Z$ 
generated by the set $\{(i,j)\}$
where $i,j$ lie between $4$ and $n+3$. 
Alternatively 
as in the proof of
Theorem~\ref{theorem:t+1arcGamman+3},
$Z=\langle \{(4,5), (4,6),\ldots,(4,n+2),(4,n+3)\}\rangle$. 
\end{proof}

%
%
%
%
\section{A fourth proof of an extended Desargues Theorem.}
\label{section:fourthProof}
\begin{theorem}
  Let $\mathbf{A},\mathbf{B}$ be simplexes in $\Sigma_n$
  which are in perspective from a point and share no points or faces.
  Then the intersections of corresponding edges of
  $\mathbf{A},\mathbf{B}$ lie in a hyperplane of $\Sigma_n$.
\end{theorem}
\begin{proof}
We use the method in~\cite{conway} in the case $n=2$.
In detail, let $A_1,A_2,A_3$ and $B_1,B_2,B_3$ form triangles
in the plane $\pi$ which are in perspective from $V$
and share no vertices or edges.
Choose any point $W$ of the 3-space containing $\pi$
that is not in $\pi$.
Let $A_2^*\ne W,A_2$ be a point on the line $WA_2$.
Then $A_2^*$ is not in $\pi$.
The plane $\sigma$ formed from $W,A_2,B_2$ also contains
$A_2^*$ and $V$.

We define $B_2^*$ as the intersection of $VA_2^*$ and $WB_2$.
The triangles $A_1,A_2^*,A_3$ and $B_1,B_2^*,B_3$
are in perspective from $V$ and do not lie in a plane.
Thus, the intersections of corresponding lines of these two triangles
lie in a line $l^*$ which is the intersection of the
planes of the two triangles.

$W$ projects the two triangles to the plane $\pi$.
It projects lines $A_1A_2^*, B_1B_2^*$ to lines $A_1A_2,B_1B_2$.
The intersection of $A_1A_2^*$ and $B_1B_2^*$
is projected by $W$ to the point of intersection in $\pi$
of lines $A_1A_2,B_1B_2$.
Similarly the intersection point of lines $A_2^*A_3,B_2^*B_3$
is projected to the intersection point of lines $A_2A_3$
and $B_2B_3$.
Since $A_1,A_3,B_1,B_3$ are in $\pi$
the point $A_1A_3\cap B_1B_3$ is projected to itself.
In summary, $W$ projects $l^*$ to the line $l$ in $\pi$
containing the 3 points of intersection of corresponding
lines of the triangles $A_1A_2A_3$, $B_1B_2B_3$.
This proves the planar theorem.

In $\Sigma_n$ we have two simplexes $\mathbf{A},\mathbf{B}$
in perspective from a point $V$ with
$\mathbf{A}=\{A_1,A_2,\ldots,A_{n+3}\}$,
$\mathbf{B}=\{B_1,B_2,\ldots,A_{n+3}\}$.
$\mathbf{A},\mathbf{B}$ share no points or faces.
As above, a point $W$ not in $\Sigma_n$ lifts
$A_2,B_2$ to points $A_2^*,B_2^*$ in $\Sigma_{n+1}$,
not in $\Sigma_n$.
The lifted simplexes $\mathbf{A^*}=\{A_1,A_2^*,A_3,\ldots,A_{n+3}\}$
and $\mathbf{B^*}=\{B_1,B_2^*,B_3,\ldots,B_{n+3}\}$
lie in distinct hyperplanes
$H_1,H_2$ of $\Sigma_{n+1}$.
The intersections of corresponding edges lie in
$L^*=H_1\cap H_2$.
As above 
$W$ projects $L^*$ to a hyperplane $L$ of
$\Sigma_n$ containing the intersection
of corrspondinding edges of $\mathbf{A},\mathbf{B}$
in $\Sigma_n$.  This proves the theorem.
\end{proof}

%
%
%
%
\section{Concluding Remarks.}

\bigskip\noindent
{\bf
Acknowledgement:}
The author acknowledges the support of the 
National Science and Engineering Research Council of Canada 
over the last fifty years. He is also grateful to the 
National Research Council of Italy for supporting his work 
over many of these fifty years.

He thanks Professor James McQuillan of Western Illinois University 
for his insights and assistance with this work.

\bigskip
\bibliography{BruenDesarguesExtension}

\end{document}